\documentclass{amsart}
\usepackage{graphicx} % Required for inserting images
\usepackage{latexsym, enumerate}
\usepackage{amssymb}
\usepackage{amsmath}
\usepackage{amsthm, url}
\usepackage{amscd}
\theoremstyle{plain}
\newtheorem{thm}{Theorem}
\newtheorem{lem}{Lemma}
\newtheorem{defin}{Definition}
\newtheorem{prop}{Proposition}
\usepackage{tikz, tikz-cd}

\usepackage{cite}
\usepackage{color}
\usepackage{xcolor}
\usepackage[pagebackref,colorlinks,citecolor=blue,linkcolor=red]{hyperref}
  \newcommand{\Rel}{\mathrm{Rel}}
\renewcommand*{\backrefalt}[4]{\ifcase #1 (Not cited).\or (Cited p.~#2).\else (Cited pp.~#2).\fi} % Cited on...
\usepackage{float}
\usepackage{tikz}
\usepackage{lmodern}
\usetikzlibrary{decorations.pathmorphing}
\newcommand{\diam}{\mathrm{diam}}
\theoremstyle{definition}
\newtheorem{rem}{Remark}
\newtheorem{claim}{Claim}

\newtheorem{ques}{Question}

\newcommand{\Teich}{\mathrm{Teich}}
\newcommand{\PMF}{\mathbb{P}\mathcal{MF}}
\newcommand{\MCG}{\mathcal{MCG}}

\newtheorem{thmi}{Theorem}
\newtheorem{cori}[thmi]{Corollary}

\newcommand{\mathcalC}{\mathcal C}

\newcommand{\calC}{\mathcal C}
\newcommand{\calU}{\mathcal U}
\newcommand{\calV}{\mathcal V}
\newcommand{\calW}{\mathcal W}

\title{Atypical generic directions in Teichm\"uller space}
\author{Matthew Gentry Durham, Chenxi Wu, and Kejia Zhu}

\begin{document}

\begin{abstract}

Motivated by geometrically capturing generic directions in Teichm\"uller space---that is, tracking rays for random walks of the mapping class group---we use work of Chaika--Masur--Wolf and Durham--Zalloum to construct the first examples of a sublinearly-Morse Teichm\"uller geodesic rays with minimal non-uniquely ergodic vertical foliations.
\end{abstract}

%authors

% \author   {Matthew Gentry Durham}
% \address{Department of Mathematics, University of California, Riverside, CA, United States of America}
% \email{mdurham@ucr.edu}

%   \author   {Chenxi Wu}
% \address{Department of Mathematics, University of Wisconsin-Madison, Madison, WI, United States of America}
% \email{cwu367@wisc.edu}

%  \author   {Kejia Zhu}
% \address{Department of Mathematics, University of California, Riverside, CA, United States of America}
 %\email{kzhumath@gmail.com}

\maketitle

\section{Introduction}

There is a delicate interplay between the geometry of the Teichm\"uller metric on the Teichm\"uller space, $\Teich(S)$, of a finite-type surface $S$ and its Thurston compactification, $\PMF(S)$, by projectivized measured foliations.  The crux of the complexity is that Teichm\"uller geodesics are defined by deformations of singular flat structures on $S$, while $\PMF(S)$ is defined via hyperbolic geometry.  This leads to an incongruity in how the internal geometry of $\Teich(S)$ is asymptotically encoded in $\PMF(S)$.  Notably, Lenzhen \cite{lenzhen2008teichmuller} proved that geodesic rays can have limit sets in $\PMF(S)$ larger than a point (see also \cite{chaika2019limits, leininger2018limit, brock2020limit}).

In this article, we produce a concrete example highlighting how this tension manifests in the context of random walks of the mapping class group $\MCG(S)$ on $\Teich(S)$.  For our purposes, the main question here is the following:
\begin{ques}
    Can \emph{generic directions} in $\Teich(S)$---namely geodesic rays tracked by sample paths of random walks---be completely characterized via properties of the internal geometry of $\Teich(S)$?
\end{ques}

In their breakthrough paper, Kaimanovich--Masur \cite{kaimanovich1996poisson} proved that Thurston's compactification of Teichm\"uller space is a topological model for the Poisson boundary for many random walks on the mapping class group, with sample paths tracking Teichm\"uller geodesic rays with \emph{uniquely ergodic} vertical foliations.  One upshot here is that the limit set in $\PMF(S)$ of a generic direction in $\Teich(S)$ is a unique point, hence avoiding the above tension between $\Teich(S)$ and $\PMF(S)$.

More recently,  Gekhtman--Qing--Rafi \cite{gekhtman2022genericity} proved that these tracking rays are \emph{sublinearly Morse} \cite{qing2022sublinearly, qing2024sublinearly}, a weak hyperbolicity condition which is controlled by the internal geometry of $\Teich(S)$ (see Durham--Zalloum \cite{durham2022geometry}).  In \cite{gekhtman2022genericity}, the authors proved that the set of accumulation points in $\PMF(S)$ of all sublinearly Morse Teichm\"uller geodesic rays has full measure with respect to any stationary measure associated to a (sufficiently nice) random walk of $\MCG(S)$ on $\Teich(S)$.  In other words, generic directions in $\Teich(S)$ are sublinearly Morse, and almost every sublinearly Morse geodesic is generic.

\smallskip

% We prove that sublinear Morseness does not imply unique ergodicity:

\begin{thmi}\label{thm:main}
    There exist Teichm\"uller geodesic rays which are sublinearly Morse but have minimal non-uniquely ergodic vertical foliations.
\end{thmi}

Our examples are on the genus two surface, but they can be lifted to higher genus examples by taking covers.  We note that \cite{kaimanovich1996poisson} proved that tracking rays are recurrent to the thick part of $\Teich(S)$, while Masur's criterion \cite{masur1982interval} implies that rays with non-uniquely ergodic vertical foliation are not recurrent.  Hence the examples Theorem \ref{thm:main} exhibit multiple non-generic pathologies.

\smallskip
The following is an interesting consequence of Theorem \ref{thm:main}:

\begin{cori} \label{cor:diverge}
    There exist divergent Teichm\"uller geodesic rays with the same underlying topological vertical foliations which diverge at a sublinear rate.  
\end{cori}

\begin{proof}
    In our examples, the simplex of measures supported by the (non-uniquely ergodic) vertical foliation is an interval.  By Remark \ref{cneq0}, we can we can construct two distinct rays $\gamma,\gamma'$ each limiting to different endpoints in $\PMF(S)$, and hence forcing $\gamma,\gamma'$ to diverge.  These rays, however, are in the same sublinearly Morse class by \cite[Theorem A, part (1)]{durham2022geometry}, which states that the sublinearly Morse boundary of Teichm\"uller space admits an injection into the Gromov boundary of the curve graph.  Hence $\gamma,\gamma'$ diverge at a sublinear rate.
\end{proof}

Our construction utilizes the robust machinery of Chaika--Masur--Wolf \cite{chaika2019limits} for producing Teichm\"uller geodesic rays with minimal non-uniquely ergodic vertical foliations.  We verify sublinear Morseness for our examples via work of Durham--Zalloum \cite[Theorem F]{durham2022geometry}, which provides a criterion for sublinear Morseness for Teichm\"uller geodesic rays.  To verify the criterion, we utilize hierarchical techniques from \cite{rafi2005characterization, rafi2014hyperbolicity, durham2023cubulating} to analyze how quickly the shadows of our geodesic rays diverge in the curve graph $\calC(S)$ of $S$.  In particular, we show that certain examples from \cite{chaika2019limits} achieve a balance of diverging quickly enough in $\calC(S)$ to be sublinearly Morse but not so quickly that they are uniquely ergodic.  See \cite{cheung2007slow, trevino2014ergodicity, chaika2017logarithmic} for the connection between divergence rates and ergodicity.

Theorem \ref{thm:main} is optimal in the following sense.  Sublinear Morseness for a given ray is a property controlled by some sublinear function $\kappa$.  The sublinear function in the tracking result from \cite{gekhtman2022genericity} is obtained via a non-constructive argument.  However, in \cite{qing2024sublinearly}, Qing--Rafi--Tiozzo show that sample paths of random walks of $\MCG(S)$ on itself track $\log^p$-Morse geodesics, for $p=p(S)$ controlled by the topology of the surface $S$. The sublinear function in our example is also a controlled power of $\log$, with the power arising from \cite{durham2022geometry} via a refined version (from \cite{durham2023cubulating}) of the hierarchical ``passing-up'' arguments utilized in \cite{qing2024sublinearly}.  Thus we expect that one cannot skirt around the examples in Theorem \ref{thm:main} by changing the sublinear function.

We end with an interesting consequence of Corollary \ref{cor:diverge}.  There is a natural way to associate any Teichm\"uller geodesic ray with its limit set in $\PMF(S)$.  Notably, Cordes \cite{cordes2017morse} proved that this map is a well-defined injection when restricted to the set of Morse geodesic rays.  That is, every Morse geodesic ray determines a unique limit point in $\PMF(S)$.  On the other hand, by Remark \ref{cneq0} and \cite[Theorem 2.7]{chaika2019limits}, the limit set in $\PMF(S)$ of the Teichm\"uller geodesic rays in our examples can be made to be more than a point, including possibly an interval.

Hence, even though each sublinearly Morse Teichm\"uller geodesic determines a unique point in the Gromov boundary of the curve graph by \cite[Theorem A]{durham2022geometry}, there is no such injection into $\PMF(S)$ as in \cite{cordes2017morse}.

%\subsection*{Acknowledgments}
%We would like to thank Yushan Jiang and Abdul Zalloum for useful discussions.  We would particularly like to thank Jon Chaika for providing some useful statements which helped us get started working with \cite{chaika2019limits}.  Durham was partially supported by NSF grant DMS-1906487. Wu was partially supported by Simons collaboration grant 850685.
% \subsection{Related results}

% The structure of ergodic measures of interval exchange maps is a classical problem that have been studied by Katok \cite{katok1973invariant}, Keane \cite{keane1975interval} and many others. In particular, examples of non uniquely ergodic interval exchange maps have been constructed by Katok \cite{katok1973invariant}, Sataev \cite{sataev1975number}, Keane \cite{keane1977non} and Veech \cite{veech1968kronecker}. The construction in \cite{chaika2019limits} which we utilized is derived from the construction by Veech \cite{veech1968kronecker}. Athreya-Chaika \cite{athreya2016hausdorff} studied the Hausdorff dimension of non uniquely ergodic directions in the strata $H(2)$, and these non uniquely ergodic flows have been used in Teichmuller dynamics, for example, Chaika-Smillie-Weiss \cite{chaika2020tremors} used them to constructing new horocycle orbit closures.

%Tiozzo for sublinearly tracking geodesic rays in Teich

%\newpage
\section{Informal discussion of \cite{chaika2019limits}}

We begin with an informal description of what is happening in the construction of Chaika--Masur--Wolf \cite{chaika2019limits}, which involves in part a detailed analysis of the slit torus examples due to Veech \cite{veech1969strict}.

First, take one copy of the $2$-torus and cut a slit in it.  Then lift it to a double cover with two branch points being the endpoints of the slit.  This cover is a genus-$2$ surface, $S$.  The idea now is to put a particular flat structure on each torus and then deform it.  This deformation will lift under the cover to a geodesic in the Teichm\"uller space for $S$, but it will really live in the subspace corresponding to the Teichm\"uller space of the twice-punctured torus.

Picking any flat structure on the torus and cutting the slit determines a point in $\Teich(S)$, as well as a quadratic differential.  So we want to pick a particular flat structure so that the ``vertical'' direction is skewed.  This is the meaning of the column vectors in the matrix
$$\begin{pmatrix}1&-\alpha\\0&1\end{pmatrix}.$$

This choice of skewed vertical direction determines a lattice in the plane $\mathbb{R}^2$.  The idea then is to choose a pair of points in the lattice that will correspond to the next vertical and horizontal directions that we want.  Flowing along the Teichm\"uller geodesic flow, we can arrange that the lengths of these new vertical and horizontal lines become comparable to $1$, while the original slit becomes very long.

Now we have a very long slit on the torus, with one endpoint at the corner point (all four corners are identified), while the other point is somewhere in the interior.  The last claim follows from arranging the original vertical direction to be skewed at an irrational angle, and also the flow time being irrational, so that the other endpoint is not at the corner.  This is the point of Proposition \ref{Prop1}.

With this setup, we can draw an arc from the interior endpoint to the corner, whose length is bounded by the area of the torus.  This new slit now lifts to a curve which has bounded length in the corresponding metric on $S$.  We repeat this process, flowing to the next pair of vertices, so that the second slit now has a very long length with one endpoint at the corner and the other in the interior.

Through the work of \cite{chaika2019limits}, we gain explicit control over the slit curves appearing through this process from the information encoded in the continued fraction expansion of $\alpha$.  However, we will only need the output of their analysis.

\section{Our Teichm\"uller geodesic and its slit curves}\label{sec:Teich}

In this section, we extract some basic information about the slit curves arising in the construction from \cite{chaika2019limits}.  We begin by setting up some notation.

\subsection{Setup and notation}

In this subsection, we will introduce some of the key players in the paper.  We point the reader toward \cite{wright2015translation} and \cite{farb2011primer} for background on translation surfaces and Teichm\"uller theory, respectively.

Let $\alpha:=[1,4,9,16,25,\dots,k^2, \dots]$ denote the real number with given continued fraction expansion, here by $\alpha=[a_1, a_2, a_3, \dots]$ we mean
\[\alpha=\frac{1}{a_1+\frac{1}{a_2+\frac{1}{a_3+\frac{1}{\dots}}}}\]
and let $\frac{p_k}{q_k}$ denote the $k^{th}$ convergent for $\alpha$ from that expansion, i.e. 
\[\frac{p_k}{q_k}=[1, 4, 9, \dots, k^2]=\frac{1}{1+\frac{1}{4+\frac{1}{9+\frac{1}{\dots+\frac{1}{k^2}}}}}\]

\begin{rem}\label{lem:rot est}
    It follows from the definition that $q_{k+1}=(k+1)^2q_k+q_{k-1}$.
\end{rem}

Following the notation of \cite[Section 2]{chaika2019limits}, consider a subsequence $\{\frac{p_{n_k}}{q_{n_k}}\}$ defined by setting $n_k=2k+1$. Now as in \cite[Section 2.1]{chaika2019limits}, let 
\[b=2\sum_{k=1}^\infty \left(q_{n_k}\alpha-\lfloor q_{n_k}\alpha \rfloor\right)=2\sum_{k=1}^\infty\left(q_{2k+1}\alpha-p_{2k+1}\right)\]
Let $T$ denote the horizontally and vertically oriented square torus and 
\[Y=\begin{pmatrix}1&-\alpha\\0&1\end{pmatrix}T\]
We glue two isometric, identically oriented copies of $Y$ together along a slit with holonomy $(b,0)$. Denote this flat surface by $X$. Let 
\[g_t=\begin{pmatrix}e^t&0\\0&e^{-t}\end{pmatrix}\]
then the action of $g_t$ on $X$ induces a Teichm\"uller geodesic $\gamma$. Observe that the first return of the vertical flow to the horizontal base is a rotation by $\alpha$.

At the beginning of Section 2 of \cite{chaika2019limits}, the authors introduced three assumptions on the number $\alpha$, which we will now verify

\begin{lem}\label{assumpsabc}
The following hold:

\begin{enumerate}[(A)]
    \item $\sum_{k=1}^\infty (a_{{n_k}+1})^{-1}<\infty$: 
    \item $a_{n_k}\to\infty$: 
    \item $q_{n_{k-1}}\log a_{n_k+1}=o(q_{n_k})=o(a_{n_k}q_{n_k-1}+q_{n_k-2})$. 
\end{enumerate}    
\end{lem}

\begin{proof}
    For item (A), observe that
    \[\sum_{k=1}^\infty (a_{{n_k}+1})^{-1}=\sum_{k=2}^\infty (2k)^{-2}<\infty.\]

    For item (B), observe that $a_{n_k}=(2k+1)^2$ which converges to $\infty$ as $k\to\infty$.

    Finally, for item (C), Remark \ref{lem:rot est} provides that
    \[\frac{q_{n_{k-1}}\log a_{n_k+1}}{q_{n_k}}=\frac{q_{2k-1}\log(a_{2k+2})}{q_{2k+1}}=O\left(\frac{\log(k)}{k^4}\right)=o(1).\]
\end{proof}

Recall that a measure preserving flow on a space $Y$ with probability measure $m$ is said to be \emph{ergodic} if $Y$ cannot be written as the disjoint union of two subsets that are invariant under the flow and each of positive measure. The flow is \emph{uniquely ergodic} if $m$ is the unique invariant measure for the flow. %(It follows that m is ergodic.)

\begin{lem}\label{LemmaNUE}
   The Teichm\"uller geodesic $\gamma$ associated with $X$ is non-uniquely ergodic.
\end{lem}

\begin{proof}
This is an immediate consequence of \cite[Theorem 2.3]{chaika2019limits}, which was originally proved by Veech in \cite{veech1969strict}, which says that if $\alpha$ satisfies Assumption (A) in Lemma \ref{assumpsabc}, then the vertical flow as an interval exchange map is non-uniquely ergodic. And Assumption (A) has been verified in Lemma \ref{assumpsabc}.
\end{proof}

% \item $\frac{p_{2i-1}}{q_{2i-1}}<\frac{p_{2i+1}}{q_{2i+1}}<\alpha<\frac{p_{2i+2}}{q_{2i+2}}<\frac{p_{2i}}{q_{2i}}$ for all $i\geq 1$.
% \item $\frac 1 {q_i(q_{i+1}+q_k)}<|\alpha-\frac{p_i}{q_i}|<\frac 1 {q_iq_{i+1}}$ and so
% \item $\frac1 {((i+1)^1+2)q_i}<\frac 1 {q_{i}+q_{i+1}}<|q_i\alpha-p_i|<\frac 1 {q_{i+1}}<\frac 1 {(i+1)^2q_i}$
% \item $Det\begin{pmatrix}p_{i}&p_{i+1}\\q_i&q_{i+1}\end{pmatrix}=1$

% {\sc check that the columns are in the correct order}

\subsection{The geodesic ray and its slit curves}

In this subsection, we begin our analysis of the geodesic ray $\gamma$.  Roughly speaking, we need a sequence of times $\{t_k\}$, so that the gap between $t_k, t_{k+1}$ grows logarithmically and so that at each time $t_k$ there is an increasingly short \emph{slit curves} $\zeta_k$, in the complement of each of which there are no short curves.  In Section \ref{sec:curve graph}, we show that this sequence of slit curves diverges in the curve graph $\mathcalC(S)$ at a linear rate (in the index of the sequence).

\medskip
Our first Proposition \ref{Prop1} provides the sequence of times corresponding to the slit curves on the geodesic ray we need.  Proposition \ref{propfill} then shows that the slit curves whose indices are sufficiently far apart are filling.  These results setup Proposition \ref{Prop2} in Section \ref{sec:curve graph}, which explains how the slits curves spread out at a uniform rate in the curve graph $\mathcalC(S)$.

\begin{prop}\label{Prop1}
Let $t_k:=\log (q_{2k+1})$,
    \begin{enumerate}
    \item For any fixed $D>0$, there exists a constant $C>0$ so that for every $k$:
    \[|t_k - t_{k+D}| < C\log t_k.\]
    \item The geodesic flow of $X$ at time $t_k$ can be split into two tori $Y^k_{\pm}$, each with a pair of slits as the boundary, and both are uniformly $\epsilon$-thick, where $\epsilon$-thick means systolic lengths $\ge\epsilon$.
    \item The (hyperbolic) lengths of the splitting slits $\zeta_k$ at time $t_k$ goes to $0$ as $k\to \infty$.
\end{enumerate}
\end{prop}
\begin{proof}
Item (1): Since $t_k=\log (q_{2k+1})$,
\[|t_k - t_{k+1}|=|\log (q_{2k+1}/q_{2k+3})|.\]
By Remark \ref{lem:rot est}, $q_{k+1}=(k+1)^2q_k+q_{k-1}$, so
\[|t_k - t_{k+1}|= \log\left(\frac{(2k+1)^2(2k)^2q_{2k-1}+q_{2k-1}+(2k+1)^2q_{2k-2}}{q_{2k}}\right)=O(\log(k)).\]
Moreover, by
\[|t_k - t_{k+D}|\le |t_k - t_{k+1}|+\dots+|t_{k+D-1} - t_{k+D}|\]
it follows that
\[|t_k - t_{k+D}|=O(\log(k)).\]

To see that $\log k < \log t_k$, recall that $\frac{p_k}{q_k}$ converges to the continued fraction for $\alpha=[1,4,9,16,...]$ (hence each $q_i>0$).  Also recall that $t_k:=\log (q_{2k+1})$, so
\[q_{k}=k^2q_{k-1}+q_{k-2}>k^2q_{k-1}\]
and
\[\frac{q_{k}}{q_{k-1}}>k^2\]
Therefore, $q_k>\prod_{i=1}^k i^2$, and hence
\[t_k>\log (q_{k})>\sum_{=1}^k\log (i^2)>k\]
In particular, $\log(t_k)>\log(k)$.\\

Item (2): This follows immediately from Proposition 4.2 of \cite{chaika2019limits}. %Lemmas 2.18 and 2.19 of \cite{chaika2019limits} show that at time $t_k$ the curves $\sigma^k_{\pm}$ and $\beta^k_{\pm}$, which represent bases of $H_1$ of the tori, are disjoint from the slit curves $\zeta^k$, and hence live in the complementary once-punctured tori. Since we are in the symmetric case of the construction, Proposition 4.2 of \cite{chaika2019limits} applies, which says that $\sigma^k_{\pm}$ and $\beta^k_{\pm}$ are almost vertical and almost horizontal respectively, while it also says that as $k$ increases, the tori $X^k_{\pm}$ converges quickly to the standard flat tori.  Hence the lengths of $\sigma^k_{\pm}$ and $\beta^k_{\pm}$ are uniformly $\epsilon'$-bounded from below by some $\epsilon'>0$, and they also fill the tori. Moreover, the angles of $\sigma^k_{\pm}$, $\beta^k_{\pm}$ are bounded from below because the area of each torus is always $1$, while the lengths of $\sigma^k_{\pm}$ and $\beta^k_{\pm}$ are uniformly bounded below. So the systoles are bounded below by some $\epsilon$, so $S_{\pm}^k$ are uniformly thick.\\

%Item (3) is part of the statement of Proposition 6.5 on page 26 of CMW, and those conditions are checked on page 34 in Subsection 7.1 in proof of Theorem 2.7.

Item (3): By \cite[Lemma 2.16]{chaika2019limits}, if we set $n_k=2k+1$,
%(similar to the proof of Theorem 2.7, p.34, note this corresponding part doesn't use the assumption that $c\not=0$, so it works well in our case here,
then the flat length of each slit $\zeta_k$ at time $t_k$ satisfies $$|\zeta_k|\asymp \frac{1}{a_{n_k+1}}.$$
(Also see the proof of Theorem 2.7 of \cite[Section 7.1]{chaika2019limits}.)
Therefore, $|\zeta_k|\to 0,$ as $k\to\infty$.
For each slit $\zeta_k$, consider a disk $D_k$ in the torus containing the $\zeta_k$, which is indeed a cylinder. Gluing along the slit of each cylinder gives an annulus. By the Riemann Mapping Theorem, this cylinder is conformal to a flat cylinder $C_k$. %We assume each $C_k$ is the biggest.
%Note that as each torus by Item(2) is not degenerate, the (biggest) flat cylinder $C_k$ may approximately sustain the area.

Now, since the flat length $|\zeta_k|$ of $\zeta_k$ is approaching $0$, the modulus of a cylinder $C_k$, say $m_k$, must be diverging to $\infty$ with $k$. So the extremal length $E(\zeta_k)$---which is, by definition, the reciprocal of the modulus of the biggest cylinder---goes to $0$, as $k$ goes to infinity. Now by Maskit's comparison result \cite{maskit1985comparison}, we have $$\frac{H(\zeta_k)}{\pi}\le E(\zeta_k). $$
Therefore, the hyperbolic length of the slit $\zeta_k$, $H(\zeta_k) $, goes to $0$, as $k$ goes to infinity.
\end{proof}

%$b_k/h_k$ is approaching $0$, where $b_k$ is the base length of the flat cylinder, $h_k$ is the height length. Recall the modulus of a cylinder $C_k$, say $m_k$, is just the $\frac{h_k}{b_k}$, hence the modulus of the biggest cylinders $m_k\to\infty$ as ${k\to\infty}$. Note the extremal length $E(\zeta_k)$ is the reciprocal of the modulus of the biggest cylinder.

\begin{figure}[h]
    \centering
\begin{tikzpicture}[scale=1.2]

\draw (-1.6,0) -- (0.4,0);
\draw [->](-1.6,0) -- (-0.8,0);
\draw [->](-2,2) -- (-1,2);
\draw (-2,2) -- (0,2);
\draw (-1.6,0) -- (-2,2);
\draw (0.4,0) -- (0,2);
\draw [->](0.4,0) -- (0.2,1);
\draw [->](-1.6,0) -- (-1.8,1);
%\draw (-0.8,1.2) -- (-0.8,0.8);
\draw (-0.6,1.0) -- (-1.0,1.0);
\draw[fill=none][red](-0.8,1) circle (0.4);
\draw [green][dotted](-1.6,0) -- (-1.6,2);
\draw [green][thick](-2,2) -- (-1.6,2);

\draw (-4.1,0) -- (-2.1,0);
\draw [->](-4.1,0) -- (-3.1,0);
\draw (-4.5,2) -- (-2.5,2);
\draw [->](-4.5,2) -- (-3.5,2);
\draw (-4.1,0) -- (-4.5,2);
\draw (-2.1,0) -- (-2.5,2);
\draw [->](-2.1,0) -- (-2.3,1);
\draw [->](-4.1,0) -- (-4.3,1);
\draw (-3.5,1.0) -- (-3.1,1.0);
 \draw[fill=none][blue](-3.3,1) circle (0.4);

\draw [blue](1.5,1.5) to[out=0,in=-0] (1.5,0.5);
\draw [blue](1.5,1.5) to[out=180,in=-180] (1.5,0.5);

\draw [red](4.5,1.5) to[out=0,in=-0] (4.5,0.5);
\draw [red][dotted](4.5,1.5) to[out=180,in=-180] (4.5,0.5);

\draw (3,1.3) to[out=0,in=-0] (3,0.8);
\draw [dotted](3,1.3) to[out=180,in=-180] (3,0.8);

\draw (1.5,1.5) to[out=0,in=180] (3,1.3);
\draw (1.5,0.5) to[out=-0,in=-180] (3,0.8);
\draw (3,0.8) to[out=-0,in=-180] (4.5,0.5);
\draw (3,1.3) to[out=0,in=180] (4.5,1.5);

%\filldraw  (6,2) circle (2pt);

%\node at (10,1.7) {$a_6$};

\end{tikzpicture}
\caption{On each torus, find a disk containing the slit; gluing the two disks containing the slit along the slits gives the cylinder. The length of the green segment is the shear of the torus, $\alpha$.}
\label{fig:cylinder}

\end{figure}

\begin{rem}\label{cneq0}
    By Lemma \ref{LemmaNUE}, the vertical flow of the flat surface $X$ defined above is non uniquely ergodic. The flat structure of $X$ corresponds to the average of the two ergodic measures ($c=0$ in the language of \cite{chaika2019limits}). If we replace the flat structure with one which corresponds to a weighted sum of the two ergodic measures, where weights are non-equal but both positive ($c\in (-1, 1)$ and $c\not=0$ in the notation of \cite{chaika2019limits}), then for this new flat surface $X_c$, the conclusions of Proposition \ref{Prop1} are still valid. Because the proof of sublinearly Morseness in the later sections only depends on Proposition \ref{Prop1}, we can pick the Teichmuller geodesic whose limit is an interval in $\PMF$, and the interval depends on the parameter $c$ (See \cite[Theorem 2.7]{chaika2019limits}, which requires that $\alpha$ satisfies Assumptions (A), (B) and (C) in Lemma\ref{assumpsabc}). In particular:
    \begin{itemize}
        \item Item (1) only involves the definition of $t_k$ and is unrelated to the surface itself.
        \item Item (2) can be shown by replacing Proposition 4.2 with Proposition 4.5 of \cite{chaika2019limits} in the proof of Proposition \ref{Prop1}.
        \item Item (3) is due to the fact that the horizontal component of the slit for $X_0=X$ is the average of the horizontal component of the slit for $X_c$ and for $X_{-c}$, hence the latter is no more than twice that of the slit of $X_0$. The vertical components of these slits are all the same. Now apply the same argument as in Proposition \ref{Prop1} to get an estimate of the hyperbolic length. Also see the proof of Theorem 2.7 in Section 7 of \cite{chaika2019limits}.
    \end{itemize}
\end{rem}

To set ourselves up for our combinatorial arguments in Section \ref{sec:curve graph}, we need another proposition.

\begin{prop}\label{propfill}
    When $k$ is sufficiently large, $\zeta_{k+3}$ and $\zeta_{k-3}$ fill the surface $X$.
\end{prop}

\begin{proof}
    As before, denote the geodesic ray by $\gamma$. By \cite[Lemma 2.16]{chaika2019limits}, when $k$ is large enough, at time $t_{k+3}$, the horizontal length of the slit $\zeta_{k+3}$
    \[h(\zeta_{k+3})\simeq 1/a_{n_{k+3}+1}=1/a_{2k+8}\]
    (where $\simeq$ means bounded both above and below up to a multiplicative constant depending only on the surface), %\marginpar{Chenxi: Please change this notation.}
    while the vertical length of the slit $\zeta_{k+3}$
    \[v(\zeta_{k+3})\simeq q_{n_{k+2}}/q_{n_{k+3}}=q_{2k+5}/q_{2k+7}\]
    So, at time $t_{k+1}$, the vertical length of $\zeta_{k+3}$ would be much larger than $1$, hence would cross both torus vertically many times. In other words, for any point on the geodesic flow of $X$ at time $t_{k+1}$, the shortest horizontal segment from this point to $\zeta_{k+3}$ has length bounded by $O(1)$, We call this length the ``horizontal distance''. See Figure \ref{fig:hordist}.

    \begin{figure}[h]
    \centering
\begin{tikzpicture}[scale=1.5]

\draw (-1.6,0) -- (0.4,0);
\draw [->](-1.6,0) -- (-0.8,0);
\draw [->](-2,2) -- (-1,2);
\draw (-2,2) -- (0,2);
\draw (-1.6,0) -- (-2,2);
\draw (0.4,0) -- (0,2);
\draw [->](0.4,0) -- (0.2,1);
\draw [->](-1.6,0) -- (-1.8,1);
%\draw (-0.8,1.2) -- (-0.8,0.8);
\draw (-0.6,1.0) -- (-1.0,1.1);
\draw[blue, dashed](-1, 1.1)--(-1.1, 2);
\draw[blue, dashed](-0.7, 0)--(-0.9, 2);
\draw[blue, dashed](-0.5, 0)--(-0.6, 1.0);
\draw[green, thick](-0.85, 1.5)--(-0.4, 1.5);
\node at (-0.3, 1.5) {$P$};
\end{tikzpicture}
\caption{The blue dashed line represents the $\zeta_{k+3}$, and the length of the green segment represents the horizontal distance from some point $P$ to $\zeta_{k+3}$.}
\label{fig:hordist}

\end{figure}

    Hence, at time $t_k$, the shortest horizontal distance from every point to
    \[\zeta_{k+3}\simeq q_{n_{k}}/q_{n_{k+1}}=O(1/k^4)\]
    By a similar argument, at time $t_k$ the shortest vertical distance from every point to $\zeta_{k-3}$ is bounded by $O(1/k^4)$. Because the slits are geodesics and the flat surface is non-positively curved, they intersect minimally. If there is some non-trivial loop on the surface that is disjoint from both slits, it has to travel outside a small neighborhood of the slit. However, by the argument above, at time $t_k$, outside a small neighborhood of the slit both tori are cut by $\zeta_{k-3}$ and $\zeta_{k+3}$ into a grid of small rectangles of size $O(1/k^4)$, so such a non-trivial loop is impossible. This shows that $\zeta_{k-3}$ and $\zeta_{k+3}$ fill the surface.
\end{proof}

%Note we glue two isometric, identically oriented copies of $Y$ together along a slit with holonomy $(\sum_{j=1}^\infty 2( p_{2j}-q_{2j}\alpha),0)$. Denote this surface by $S$. So the initial length of the slit is $\sum_{j=1}^\infty 2( p_{2j}-q_{2j}\alpha$. The

\section{Combinatorics of the slit curves}\label{sec:curve graph}

The main goal of this section is Proposition \ref{Prop2} below, in which we prove that the slit curves $\zeta_k$ provided by Proposition \ref{Prop1} make linear progress (in $k$) in the curve graph $\mathcal C(S)$.  See \cite{minsky2006curve, MM_II} for background on curve graphs and subsurface projections.

\begin{prop}\label{Prop2}
For any $L>0$, there exists a sequence ${k_n}$ and a constant $B = B(S,L)>0$ so that we have
\begin{enumerate}
    \item $|t_{k_n} - t_{k_{n+1}}| < B \log k_n$.
    \item $d_{\mathcal C(S)}(\zeta_{k_n}, \zeta_{k_{n+1}}) >L.$
\end{enumerate}
\end{prop}

In Lemma \ref{lem:spacing} below, we derive a more refined version of Proposition \ref{Prop2}, but the proposition above is the key result.

The proof involves an iterated argument which a blend of some results connecting short curves and subsurface projections due to Rafi \cite{rafi2005characterization} and so-called ``passing-up'' techniques for producing large subsurface projections satisfying certain properties due to Durham \cite{durham2023cubulating}.

\begin{rem}
    We note one of the main technical difficulties here is that the slit curves $\zeta_k$ need not correspond to boundary curves of subsurfaces with large projections for $\mu^-, \mu^+$.  In particular, knowing that they are pairwise filling (Proposition \ref{propfill}) alone is not enough to show that they spread out in $\calC(S)$ at a uniform rate.
\end{rem}

\begin{rem}
    In the remaining proofs in this paper, when we write that some quantity is ``uniform'' or ``uniformly bounded'', we mean that the quantity in question is controlled by the topology of $S$.
\end{rem}

In the proof of Proposition \ref{Prop2}, we will frequently use the following results due to Rafi \cite{rafi2005characterization} and Durham \cite{durham2023cubulating}, which we include here for clarity.  Please see \cite{MM_I,MM_II, rafi2005characterization, rafi2007combinatorial, durham2016augmented} for background on subsurface projections and Teichm\"uller geometry.

The first statement we need is \cite[Theorem 6.1]{rafi2005characterization}.  It says that short curves along Teichm\"uller geodesics determine large subsurface projections in their complement, and vice versa.  The version below is quantitative---the shorter the curve, the larger the projection---and follows from the Rafi's original proof in \cite{rafi2005characterization}, though it is not commonly stated this way in the literature.  We need this quantitative version to know that the increasingly shorter slit curves produce increasingly larger projections, which we can then plug into the hierarchical machinery we will describe next.

\begin{prop}[Short curves and big projections]\label{prop:short curves}
Given a finite-type surface $S$, there exists constants $\epsilon_0 = \epsilon_0(S)>0$ and $K_0=K_0(S)>0$  so that for any Teichm\"uller geodesic $(g,\mu^-, \mu^+)$ in $\Teich(S)$, the following hold:

\begin{enumerate}
    \item For any $K\geq K_0$, there exists $\epsilon_{K}>0$ so that if $\alpha$ is a simple closed curve on $S$ so that $l_{g(t)}(\alpha)<\epsilon_K$ for some time $t$, then there exists a subsurface $Y \subset S - \alpha$ with $d_Y(\mu^-,\mu^+)>K$ (or $\log d_{Y}(\mu^-,\mu^+)> K$ in case $Y=Y_{\alpha}$ is the annulus about $\alpha$). \label{item:short to big}
 
    \item For any $\epsilon\geq\epsilon_0$, there exists $K_{\epsilon}>0$ so that if $d_Y(\mu^-,\mu^+)>K_{\epsilon}$ (or $\log d_Y(\mu^-,\mu^+)>K_{\epsilon}$ when $Y$ is an annulus), then there exists a time $t$ and a simple closed curve disjoint from $Y$ so that $l_{g(t)}(\alpha)<\epsilon$ for some time $t$. \label{item:big to short}
\end{enumerate}
\end{prop}

The second statement we need is \cite[Proposition 4.7]{durham2023cubulating}.  Roughly, it states that the boundary curves of a large collection $\calU$ of subsurfaces to which a Teichm\"uller geodesic $g$ has a large projection $\pi_W(g)$ must spread out along the projection of $g$ to the curve graph $\calC(W)$ of some subsurface $W \subseteq S$.  Moreover, this statement is effective---the larger the cardinality of $\calU$, the more of $\pi_W(g)$ they cover.  We need it to gain quantitative control over how quickly the slit curves in our example spread out in the curve graph of the whole surface, $\calC(S)$.

We need some terminology to set up the statement.  Given a Teichm\"uller geodesic $(g,\mu^-,\mu^+)$ and a constant $K>0$, we say that $Y \subset S$ is $K$-\emph{relevant} if $d_Y(\mu^-,\mu^+)>K$.  We let $\Rel_K(\mu^-,\mu^+)$ denote the set of $K$-relevant domains for $g$.

We also need the notion of a $\sigma$-subdivision, which will allow us to precisely define what it means for boundary curves to spread out along the projection of a Teichm\"uller geodesic.

\begin{defin}[$\sigma$-subdivision] \label{defn:sigma subdivision}
Let $W \subseteq S$ and $\gamma:I \to \calC(W)$ be a geodesic path (i.e., segment, ray, or bi-infinite) in $\calC(W)$ between $\pi_W(\mu^-),\pi_W(\mu^+)$ coming from a geodesic $(g,\mu^-,\mu^+)$.  For $\sigma>0$, we say that a subdivision $\{x_i\}$ of $I$ is a $\sigma$-\emph{subdivision} of $\gamma$ if the $x_i$ decompose $I$ into a collection of subintervals $[x_i, x_{i+1}]$ so that for all but at most one $i$ we have $|x_{i+1}-x_i| = \sigma$, with the (possibly nonexistent) extra subinterval for which $|x_{i+1} - x_i| < \sigma$.
\end{defin}

Given a $\sigma$-subdivision $\{x_i\}$ of a geodesic $\gamma$ between $\pi_W(\mu^-),\pi_W(\mu^+)$ in $\calC(W)$ for $W \subseteq S$ and a collection of relevant subsurfaces $\calV$ of $W$, we let $\calV_i$ denote the set of domains $V \in \calV$ so that $p_{\gamma}(\partial V) \cap [x_i,x_{i+1}] \neq \emptyset$, where $p_{\gamma}$ denotes closest point projection in $\calC(W)$ to $\gamma$.  The following proposition is \cite[Proposition 4.7]{durham2023cubulating} specialized to our setting:

\begin{prop}[Strong passing-up]\label{prop:SPU}
    There exists $E=E(S)>0$ so that for any $K_2 \geq K_1 \geq 50E$, there exists $P_1 = P_1(K_1, K_2)>0$ so that if $(g,\mu^-,\mu^+)$ is a Teichm\"uller geodesic in $\Teich(S)$, the following hold.  If $\calV \subset \Rel_{K_1}(\mu^-,\mu^+)$ with $\#\calV > P_1$, then there exists $W \in \Rel_{K_2}(\mu^-,\mu^+)$ and $\calV' \subset \calV$ so that $V \sqsubset W$ for all $V \in \calV'$ and 
$$\diam_W\left(\bigcup_{V \in \calV'} \rho^V_W\right) > K_2.$$

Moreover, for any $\sigma\geq 10E$ and $n \in \mathbb{N}$, there exists $P_2(K_1,K_2,\sigma, n)>0$ so that if $\#\calV > P_2$, then we can arrange the following to hold:
\begin{itemize}
    \item If $\gamma$ is any geodesic in $\calC(W)$ between $\pi_W(\mu^-),\pi_W(\mu^+)$ and $\{x_i\}$ is a $\sigma$-subdivision of $\gamma$ determining sets $\calV_i$ as before, then 
    $$\#\{i| \calV_i \neq \emptyset\} \geq n.$$
\end{itemize}
\end{prop}

\begin{rem}
    The original version of \cite[Proposition 4.7]{durham2023cubulating} is proven is context of \emph{hierarchically hyperbolic spaces}.  The Teichm\"uller space of any finite-type surfaces with the Teichm\"uller metric is an HHS, see \cite{rafi2007combinatorial,durham2016augmented, HHS_II}.
\end{rem}

% {\color{orange} \begin{rem}\label{remrafi}
% Modami-Rafi \cite{Modami_short} provides a quantitative version of \cite[Theorem 6.1]{rafi2005characterization}. Basically, if a simple closed curve is very short along a Teichm\"uller geodesic $\gamma$, then one can find a subsurface in its complement onto which $\gamma$ is long.
% \end{rem}}
%\textcolor{blue}{Question: how about we just cite the Propostion G instead of 4.7? Proposition G is informal but more hands-on:\\ Informally speaking, the Strong Passing-up Theorem from Durham says 
%\begin{prop}[Strong passing-up]\label{propi:SPU}
%Let $a,b$ be interior points, boundary points, or hierarchy rays in an HHS $\calX$.  Then for any sufficiently large collection $\calV$ of relevant domains for $a,b$, there exists a relevant domain $W$ and large subcollection $\calV' \subset \calV$, so that $V \sqsubset W$ for all $V$ and
%$$\diam_{\calC(W)}\left(\bigcup_{V \in \calV'} \rho^V_W\right)$$
%can be made as large as necessary by increasing the size of $\#\calV$.  Moreover, the $\rho$-sets $\rho^V_W$ distribute proportionally evenly along any geodesic between $\pi_W(a)$ and $\pi_W(b)$.
%\end{prop} For the precise statement, please see \cite[Proposition 4.7, P.38]{durham2023cubulating} for more details.}

\begin{proof}[Proof of Proposition \ref{Prop2}]

By Proposition \ref{propfill}, we may assume that the slit curves $\zeta_k$ pairwise fill $S$.  We fix a subsurface projection threshold constant $L_0 = L_0(S)>0$ to be sufficiently large to make the arguments below work.

For sufficiently large $L_0 = L_0(S)>0$, we can fix $\epsilon = \epsilon(L_0, S)>0$ to be small enough so that if $\sigma$ is a simple closed curve on $S$ so that $l_{\gamma(t)}(\sigma)<\epsilon$ for some time $t$, then {item \eqref{item:short to big} of Proposition \ref{prop:short curves} provides} a subsurface $V \subset S - \zeta$ (with possibly $V$ the annulus with core $\zeta$) so that $\mathrm{diam}_{\mathcalC(V)}(\gamma)>L_0$.

For the rest of the proof, we will increase $L_0$ and $\epsilon$ as necessary while maintaining the dependencies only on $S$.

% To begin, we first fix a subsurface projection threshold $L_0 = L_0(S)>0$ sufficiently large to satisfy the assumptions of \cite[Proposition 4.3]{durham2023cubulating}.  We will explain how to use this proposition during the argument.

Consider the sequence of slit curves $\zeta_k$.  By item (3) of Proposition \ref{Prop1}, there exists $K_0 = K_0(\epsilon)>0$ so that if $k>K_0$, then $l_{\gamma(t_k)}(\zeta_k)<\epsilon$.  For each $k>K_0$, let $V_k$ denote the $L_0$-large subsurface in the complement $S - \zeta_k$ {provided by item \eqref{item:big to short} of Proposition \ref{prop:short curves}}.  We denote the collection of these subsurfaces $V_k$ by $\calV$.

% Observe that by item (2) of Proposition \ref{Prop1}, each $V_k \in \calV$ must be one of the following:
% \begin{itemize}
%     \item The annulus around $\zeta_k$,
%     \item An annulus around a curve in $S - \zeta_k$, or
%     \item One of the one-holed tori components of $S - \zeta_k$.
% \end{itemize}

% \begin{claim}\label{claim:color}
%     There exist a constant $A = A(S)>0$ and a subsequence $\zeta_{k_n}$ so that $|k_n - k_{n+1}|<A$ for all $n$, and $V_{k_n}, V_{k_{m}}$ have pairwise intersecting boundaries for all $m\neq n$.
% \end{claim}

% \begin{proof}[Proof of Claim \ref{claim:color}]
% This follows from work of Bestvina--Bromberg--Fujiwara \cite[Proposition 5.8]{BBF}, which says that the set of all subsurfaces of our ambient surface $S$ can be partitioned into boundedly-many subcollections $\mathcal{Y}_i$ (with the bound controlled by $S$), so that if $U,V \in \mathcal{Y}_i$, then $\partial U$ and $\partial V$ intersect.  An application of the pigeon hole principle completes the proof.
% \end{proof}

% Observe that since the number of terms we need to skip is bounded when building this subsequence, this subsequence $\{k_n\}$ satisfies item (1) of Proposition \ref{Prop2}.

% We will have to similarly pass to subsequences a few times.  Each time, we will observe that the spacing still satisfies item (1) of the proposition.  To simplify notation, we will use $t_k, \zeta_k, V_k$, etc. to refer to objects indexed by these subsequences.

\medskip

Our goal is to show that the $\zeta_k$ spread out in $\mathcal{C}(S)$ uniformly quickly.  Using the $V_k$ as surrogates, the following claim forces the boundary curves $\partial V_k$ of any sufficiently large subcollection of the $V_k$ to spread out in \emph{some} subsurface, with extra control over where the slit curves lie:

\begin{claim} \label{claim:passing-up}
    For any $L_1>L_0$, there exists $N_1 = N_1(S, L_0,L_1)>0$ so that for any collection $\calU \subset \calV$ with $\#\calU \geq N_1$, there exists a subsurface $W$ and $V_j, V_j, V_k, V_l\subset W$ with $V_i, V_j, V_k, V_l \in \calU$  so that
    \begin{enumerate}
        \item $d_{\mathcalC(W)}(\mu^-, \mu^+)>L_1$, and
        \item All pairwise distances in $\calC(W)$ between $\mu^-, \mu^+, \partial V_i, \partial V_j, \partial V_k, \partial V_l$ are larger than $\frac{L_1}{100}$, and $\partial V_i, \partial V_j, \partial V_k, \partial V_l$ appear in that order along any geodesic in $\calC(W)$ between $\mu^-, \mu^+$ (see Figure \ref{fig:order}).
        % \item If there is some subsurface $Z \subset S$ so that $U \subset Z$ for each $U \in \calU$, then $W \subset Z$.
        \item At least one of the slit curves $\zeta_i, \zeta_j, \zeta_k, \zeta_l$ corresponding to $V_i, V_j, V_k, V_l$ is contained in $W$.
    \end{enumerate}

\end{claim}

\begin{figure}
\begin{center} \label{fig:order}
\begin{tikzpicture}
    \draw[-,purple] (0, 0)--(10, 0);
    \node at (-0.3, 0){$\mu^-$};
    \node at (10.3, 0){$\mu^+$};
    \node at (6, 2.5){$\calC(W)$};
    \filldraw (2,0.2) circle (1pt);
    \filldraw (2.1,-0.2) circle (1pt);
    \filldraw (4,0.2) circle (1pt);
    \filldraw (4,-0.25) circle (1pt);
    \filldraw (5.9,0.25) circle (1pt);
    \filldraw (6,-0.2) circle (1pt);
    \filldraw (8, 0.2) circle (1pt);
    \filldraw (8, -0.16) circle (1pt);
    \filldraw (0,0) circle (1pt);
    \filldraw (10,0) circle (1pt);
    \node at (2, 0.5){$\partial V_i$};
    \node at (4, 0.5){$\partial V_j$};
    \node at (6, 0.5){$\partial V_k$};
    \node at (8, 0.5){$\partial V_l$};
    \node at (2, -0.5){$\zeta_i$};
    \node at (4, -0.5){$\zeta_j$};
    \node at (6, -0.5){$\zeta_k$};
    \node at (8, -0.5){$\zeta_l$};
    \node at (1, 1.2){$>\frac{L_1}{100}$};
    \draw [blue](0.2,0.4) to[out=40,in=160] (1.6,0.7);
    \node at (3, 1.2){$>\frac{L_1}{100}$};
    \draw [blue](2.4,0.4) to[out=40,in=160] (3.6,0.7);
    \node at (5, 1.2){$>\frac{L_1}{100}$};
    \draw [blue](4.4,0.4) to[out=40,in=160] (5.6,0.7);
    \node at (7, -1){$>\frac{L_1}{100}$};
    \draw [blue](6.4,-0.4) to[out=-40,in=-140] (7.6,-0.5);
    \node at (9, 1.2){$>\frac{L_1}{100}$};
    \draw [blue](8.4,0.4) to[out=40,in=160] (9.9,0.5);

   % \draw[red] (0,0) arc (30:60:3); \draw[red] (0,0) arc (30:60:3);

\end{tikzpicture}

\end{center}
\caption{Using Proposition \ref{prop:SPU}, we can arrange for the boundary curves $\partial V_i, \partial V_j, \partial V_k, \partial V_l$, and hence $\zeta_i, \zeta_j, \zeta_k, \zeta_l$, appear in order along any geodesic in $\calC(W)$ between $\mu^-, \mu^+$.}
\end{figure}

\begin{proof}[Proof of Claim \ref{claim:passing-up}]
    Items (1) and (2) {follows immediately from Proposition \ref{prop:SPU}.  In particular, by choosing the subdivision constant to be $\sigma = \frac{1}{100}$ as in the statement, we can produce $V_i, V_j$ satisfying the desired conclusions.}

    % are essentially an immediate consequence of \cite[Proposition 4.7]{durham2023cubulating}, which not only provides a large domain $W$ containing some of the $V_k$ used to produce it, but forces some of their boundary curves to roughly evenly distribute along the geodesic in $\mathcalC(W)$ between $\pi_W(\mu^-)$ and $\pi_W(\mu^+)$.  In particular, by choosing the subdivision constant to be $\sigma = \frac{1}{100}$ as in the statement of \cite[Proposition 4.7]{durham2023cubulating}, we can produce $V_i, V_j$ satisfying the desired conclusions.

    For item (3), first observe that since the $\zeta_i, \zeta_j, \zeta_k, \zeta_l$ pairwise fill $S$ and each of $V_i, V_j, V_k, V_l \subset W \subset S$, it is not possible for any of the $\zeta_i, \zeta_j, \zeta_k, \zeta_l$ to be disjoint from $W$, for if $\zeta_i$ were disjoint from $W$, then $\zeta_j$ would have to intersect $\partial V_j$, which is impossible.  Hence they either must intersect $\partial W$, be contained in $\partial W$, or be contained in $W$ itself.  Moreover, at most one of the slit curves can be contained in $\partial W$, so we may assume, up to reindexing, that $\zeta_i, \zeta_j, \zeta_k$ are not contained in $\partial W$.

    By \cite[Theorem 5.3]{rafi2014hyperbolicity}, there is an \emph{active interval} of times $I_{V_i}$ along the Teichm\"uller geodesic $\gamma$ during which the curves in $\partial V_i$ are shorter than $\epsilon$ and outside of which $\gamma$ has a bounded diameter projection to $\calC(V_i)$, and similarly for $V_j,V_k,W$.  By our choice of the arrangement and location of $\partial V_i,\partial V_j,\partial V_k$ along any geodesic in $\calC(W)$ between $\mu^-, \mu^+$ and the fact that the projection of $\gamma$ to $\calC(W)$ is a uniform (unparameterized) quasigeodesic \cite[Theorem 6.1]{rafi2014hyperbolicity}, we may increase $L_0 = L_0(S)>0$ as necessary to arrange that $I_{V_i}, I_{V_j}, I_{V_k} \subset I_W$ and that these intervals appear in this order along the parameter interval $[0, \infty)$ for $\gamma$.

    Finally, observe that if $t \in [0,\infty)$ is such that $l_{\gamma(t)}(\zeta_j)<\epsilon$, then $t$ comes after $I_{V_i}$ and before $I_{V_k}$, again because of \cite[Theorem 6.1]{rafi2014hyperbolicity}.  Hence $t \in I_W$, and in particular $\zeta_j$ and $\partial W$ are short simultaneously along $\gamma$.  Since this is only possible when $\zeta_j$ is contained in $W$ by the Collar Lemma, this proves item (3) of the claim, as required.
\end{proof}

The rest of the argument proceeds by induction on the the complexity of the subsurfaces produced using Claim \ref{claim:passing-up} and its analogues below.  For any surface $Y$, recall that its \emph{topological complexity} is $\xi(Y) = g(Y) + p(Y)$, where $g(Y)$ counts its genus and $p(Y)$ counts its punctures.  Note that when $X \subset Y$ is a proper subsurface, we have $\xi(X) < \xi(Y)$.

We also require some more organizational notation.  Since the domains in $\calV$ come in an order provided by the order of the slit curves $\zeta_i$, we can partition $\calV$ into subcollections $\calV_j$ of consecutive domains, with each of the $\#V_j = N_1$ (for $N_1$ as in Claim \ref{claim:passing-up}), and the domains in $\calV_j$ immediately preceding those in $\calV_{j+1}$ for every $j \geq 1$.

For each $j$, Claim \ref{claim:passing-up} provides potentially many subsurfaces $W$ satisfying the conclusions of the claim for $\calU = \calV_j$.  Choose $W_j$ to be one of the subsurfaces that maximizes topological complexity among all such subsurfaces, and let $\calW$ be the collection of these $W_j$.  Given $W \in \calW$, let $J_W = \{j|W_j = W\}$.

\begin{claim}\label{claim:bounding containers}
    For any $W \in \calW$ with $W \neq S$, we have $\#J_W = 1$.
\end{claim}

\begin{proof}[Proof of Claim \ref{claim:bounding containers}]
Suppose that $W = W_i$ and $W = W_j$ for $i \neq j$.  Then item (4) of Claim \ref{claim:passing-up} provides distinct slit curves $\zeta_i \neq \zeta_j$ with $\zeta_i, \zeta_j \subset W$, which is impossible since $\zeta_i, \zeta_j$ fill all of $S$.  This completes the proof.
\end{proof}

Claim \ref{claim:bounding containers} says that the $W_i$ are all distinct.  Moreover, the $W_i$ all have topological complexity at least $3$ since each properly contains a subsurface of complexity at least $2$.

\begin{claim}\label{claim:passing-up, step 2}
    The conclusions of Claim \ref{claim:passing-up} hold when replacing the collection $\calV$ with the collection $\calW$.
\end{claim}

\begin{proof}
    Since items (1) and (2) follow from Proposition \ref{prop:SPU}, the only item to check is item (3).  However, for each $i$, we have the containments $\zeta_i, V_i \subset W_i$, so any container domain $Y$ produced by Proposition \ref{prop:SPU} necessarily satisfies $\zeta_j, V_j, \subset W_j \subset Y$ for any of the four indices produced in the claim.  This completes the proof.
\end{proof}

%\textcolor{blue}{here is not the right place, should be somewhere above, and maybe add a few more words?}

Once again partitioning the domains in $\calW$ as we did with $\calV$ above into ordered subcollections, we can use Claim \ref{claim:passing-up, step 2} to produce another collection $\mathcal{Z}$ of container domains for the $W_i$ which are of complexity at least $4$.  Moreover, it only takes $N_1$-many of the $W_i\in \calW$ to produce each such container, while each such $W_i$ only took $N_1$-many of the $V_j \in \calV$ to produce.  %Moreover, the analogue of Claim \ref{claim:bounding containers} follows immediately.

Since $\xi(S) = 6$, we can repeat this process at most two more times, at each level only increases the number of consecutive domains in $\calV$ that we need by a multiplicative factor of $N_1 = N_1(S)$.  In particular, by taking any collection $\mathcal A$ of $N_1^6$-many consecutive domains from $\calV$, we can use the appropriate version of Claim \ref{claim:passing-up} to produce domains $Z_i, Z_j$ so that
\begin{itemize}
    \item There are slit curves $\zeta_i, \zeta_j$ and $V_i, V_j$ so that $V_i \subseteq Z_i$ and $V_j \subseteq Z_j$, with either $\zeta_i \subset V_i$ (if $V_i \subsetneq Z_i$) or $\zeta_i \perp V_i$ (if $V_i = Z_i$), and similarly for $j$;
    \item $d_S(\partial Z_i, \partial Z_j) > \frac{L_0}{100}$.
\end{itemize}

Then since the slit curves $\zeta_i, \zeta_j$ project at most distance $1$ from $\partial Z_i, \partial Z_j$ in $\calC(S)$, we see that $d_S(\zeta_i, \zeta_j)> \frac{L_0}{100} - 2$.  On the other hand, since the slit curves become short along $\gamma$ in their given order and $\gamma$ projects to a uniform (unparameterized) quasigeodesic in $\calC(S)$, we have that the slit curves corresponding to first and last domains in $\mathcal A$ must also project roughly $\frac{L_0}{100}$ apart in $\calC(S)$.  This completes the proof of item (2) of Proposition \ref{Prop2}, and item (1) of this proposition follows from item (1) of Proposition \ref{Prop1}.

\end{proof}

In the following lemma, we record the information that we actually need from Proposition \ref{Prop2}.  Roughly, the lemma says that we can find a sequence of times $\{t_k\}$, whose spacing as a time parameter grows like $\log$ (in the index $k$), but whose spacing in the curve graph grows linearly (with the index $k$), with the corresponding intervals in the curve graph having bounded overlap.

We set some notation for the statement.  Recall that Masur--Minsky \cite{masur1999geometry} proved that Teichm\"uller geodesics project to unparameterized quasi-geodesics with uniform constants controlled by (the topology of) $S$.  Since $\gamma$ has a minimal non-uniquely ergodic vertical foliation, its projection $\pi_S(\gamma) \subset \mathcal{C}(S)$ is a quasi-geodesic ray.  Fix a geodesic ray $\Gamma$ in its asymptotic class\footnote{Despite the fact that $\mathcal{C}(S)$ is locally infinite, one can find geodesic ray representatives using the finiteness properties of tight geodesics \cite{bowditch2008tight}.}. For each $k$, let
$$\Gamma_k = \pi_{\Gamma}(\pi_S(\gamma|_{[t_k, t_{k+1}]}))$$
denote the closest-point projection to $\Gamma$ in $\mathcal{C}(S)$ of the projection to $\mathcal{C}(S)$ of the restriction of $\gamma$ to $[t_k, t_{k+1}]$.

\begin{lem}\label{lem:spacing}
    For our Teichm\"uller geodesic $\gamma$ and any sufficiently large $L=L(S)>0$, there exists a sequence of times $\{t_k\}$ and constants $C = C(S)>0$ so that for all $k$, we have
    \begin{enumerate}
    \item $|t_k - t_{k+1}| < C \log k$;
    \item $d_{\mathcal{C}(S)}(\gamma(t_k),\gamma(t_{k+1}))>L$;
    \item $\mathrm{diam}_{\mathcal{C}(S)}(\Gamma_k)> L/2$;
    \item If $j = k-1$ or $j=k+1$ then $\mathrm{diam}_{\mathcal{C}(S)}(\Gamma_k \cap \Gamma_{j})<L/10$;
    \item If $j \neq k-1, k, k+1$, then $\Gamma_j \cap \Gamma_k = \emptyset$.
    \end{enumerate}
\end{lem}

\begin{proof}
    Fixing $L = L(S)>0$ below, item (1) follows easily from the proof of part 1 of Proposition \ref{Prop1}.  Item (2) follows from Proposition \ref{Prop2}.  Item (3) follows from item (2) plus the fact that $\pi_S(\gamma)$ is a uniform (in $S$) unparameterized quasi-geodesic, and hence is uniformly close (in $S$) to $\Gamma$ by uniform hyperbolicity of $\mathcal{C}(S)$.

    For item (4), since $\pi_S(\gamma)$ is a uniform (unparameterized) quasi-geodesic, it can backtrack at most a bounded amount.  By hyperbolicity of $\mathcal{C}(S)$, this backtracking is efficiently recorded on the geodesic ray $\Gamma$.  In particular, by choosing $L=L(S)>0$ sufficiently large, we can guarantee that the overlap of $\Gamma_k$ with $\Gamma_{k-1}$ or $\Gamma_{k+1}$ is a controlled fraction of $L$, e.g. $\frac{L}{10}$.

    Item (5) now follows easily from items (3) and (4), completing the proof.
\end{proof}

% \textcolor{purple}{The proof isn't quite complete, because it requires slightly different statements from Propositions \ref{Prop1} and \ref{Prop2}.  We need to be able to make sure that the curve graph distance between the times $t_k$ is far enough so there is bounded (curve graph) overlap between the intervals, while the time difference between the intervals grows logarithmically.  I'm using the geodesic $\Gamma$ just as a reference object.  I think this can be obtained by what's already written.}

 %\draw[red] (0,0) arc (30:60:3);

 \section{Confirming sublinear Morseness}

In this final section, we confirm that our examples are sublinearly Morse.  For our purposes, the definition of sublinearly Morse is not strictly necessary, though we include it below; see \cite{qing2024sublinearly, durham2022geometry} for more background in our setting.

Just for the sake of completeness, we recall the definition of sublinearly Morseness in \cite[Definition 3.2]{qing2024sublinearly}, also see \cite[Definition 2.3]{durham2022geometry}:

A sublinearly Morse geodesic ray $\gamma$ satisfies a sublinear version of the Morse property satisfied by geodesics in hyperbolic spaces.  Roughly speaking, this says that quasigeodesics that start and end on $\gamma$ must stay sublinearly close, for some uniform sublinear function.

Let $\kappa: [0, \infty)\to [1, \infty)$ be a sublinear function, i.e. $\displaystyle \lim_{t\rightarrow\infty}\kappa(t)/t=0$.  Let $X$ be a proper geodesic metric space and fix a basepoint $\mathfrak{o} \in X$.

\begin{defin}
A closed set $Z\subseteq X$ is called {\em $\kappa$-Morse} if there is a proper function $m_Z: \mathbb{R}^+\times\mathbb{R}^+\rightarrow\mathbb{R}^+$, such that for any sublinear function $\kappa'$, any $r>0$, there exists $R\geq r$, such that for any $(q, Q)$-quasi-geodesic ray $\beta$ with $\beta(0)=\mathfrak{o}$, and $m_Z(q, Q)$ sufficiently small compared to $r$, we have
\[d_X(\beta(R), Z)\leq\kappa'(R)\implies \text{ for any }t\in [0, r], d(\beta(t), Z)\leq m_Z(q, Q)\kappa (d_X(\mathfrak{o},\beta(t)))\]
\begin{itemize}
    \item The function $m_Z$ is called a {\em Morse gauge} of $Z$.
\end{itemize}
\end{defin}

The following is a useful criterion from \cite{durham2022geometry} for a Teichm\"uller geodesic to be sublinearly-Morse.

\begin{defin}
    For a given sublinear function $\kappa$, we say that a Teichm\"uller geodesic $\beta$ satisfies the $\kappa$-\emph{bounded projections} property if
    $$d_Y(\beta(0), \beta(t)) \leq C \cdot \kappa(t),$$
    for some constant $C = C(S)>0$ and for all proper subsurfaces $Y \subsetneq S$.
\end{defin}

The following is \cite[Theorem K, part 2]{durham2022geometry}:

\begin{thm}\label{thm:k-bdd}
    There exists a $p = p(S)>0$ so that if $\kappa^{2p}$ is sublinear and $\beta$ is a Teichm\"uller geodesic with $\kappa$-bounded projections, then $\beta$ is $\kappa^{2p}$-Morse.
\end{thm}\qed

% We will also need the following theorem of Masur--Minsky \cite{MM_II}:

% \begin{thm}[Bounded geodesic image]\label{thm:BGI}
%     There exists $D=D(S)>0$ so that if $\nu$ is any geodesic in $\mathcal{C}(S)$ and $Y \subsetneq S$ is a subsurface with $\nu \cap \mathcalN_2(\partial Y) = \emptyset$, then $diam_{\mathcal{C}(Y)}(\nu) < D$.
% \end{thm}

The following completes the proof of Theorem \ref{thm:main}:

\begin{thm}
    There exists $p = p(S)$ so that the Teichm\"uller geodesic $\gamma$ associated with $X$ is $\log^{2p}$-Morse.
\end{thm}

\begin{proof}
    By Theorem \ref{thm:k-bdd}, it suffices to prove that $\gamma$ has $\log$-bounded projections.  Fix a subsurface $Y \subset S$.

    As in Lemma \ref{lem:spacing} above, let $\Gamma$ be a geodesic representative in $\mathcal{C}(S)$ of the quasi-geodesic ray $\pi_S(\gamma)$.  By the Bounded Geodesic Image Theorem \cite[Theorem 3.1]{MM_II}, if $\partial Y$ is not uniformly close to $\Gamma$, then $\gamma$ has a uniformly bounded projection to $\mathcal{C}(Y)$.

    On the other hand, the closest point projection $P_Y$ of $\partial Y$ to $\Gamma$ has uniformly bounded diameter.  By increasing the constant $L$ in  Lemma \ref{lem:spacing} a bounded amount as necessary (which Proposition \ref{Prop2} allows us to do), we can guarantee that $P_Y$ overlaps at most two consecutive $\Gamma_k$, and hence that $\partial Y$ is as far as we would like from the restrictions $\gamma_k = \pi_S(\gamma|_{[t_k, t_{k+1}]})$ for all but two consecutive $k$.

    In particular, this forces the active interval $I_Y$ to overlap at most two consecutive $[t_k, t_{k+1}]$, while on the other hand the projection to $\mathcal{C}(Y)$ is coarsely constant outside of the active interval $I_Y$ for $Y$ by \cite[Theorem 6.1]{rafi2014hyperbolicity}.

    Supposing that $I_Y$ is contained in the union of $[t_{k-1}, t_k] \cup [t_k, t_{k+1}]$, then the distance formula for $\Teich(S)$ \cite{rafi2007combinatorial, durham2016augmented} and \cite[Theorem 6.1]{rafi2014hyperbolicity} provide that:
    \begin{eqnarray}
    d_Y(\gamma(0), \gamma(t)) &\asymp& d_Y(\gamma(t_{k-1}), \gamma(t))\\
    &\prec& d_Y(\gamma(t_{k-1}), \gamma(t_{k+1}))\\ &\prec& \sum_{W\subset \mathcal L_k} d_W(\gamma(t_{k-1}), \gamma(t_{k+1}))\\
    &\asymp& d_{\Teich(S)}(\gamma(t_{k-1}), \gamma(t_{k+1}))\\
    &\prec& 2 \log k\\
    &\prec& 2\log t_k\\
    &\leq& 2 \log t.
    \end{eqnarray}
    In these computations:
    \begin{itemize}
        \item The coarse (in)equalities $\asymp$ and $\prec$ are (in)equalities which hold up to bounded additive and multiplicative constants depending only on the topology of $S$.

        \item Given an annulus $V \subset S$, the distance $d_V(-,-)$ denotes distance in a combinatorial horoball over the annular curve graph for $V$ \cite{durham2016augmented} or alternatively an appropriate horoball in $\mathbb H^2$ \cite{dowdall2014statistical}.

       \item In line (3),
       \[\mathcal L_k = \{W \subset S| d_W(\gamma(t_{k-1}), \gamma(t_{k+1})>M\}\]
       denotes the subsurfaces where $\gamma(t_{k-1}), \gamma(t_{k+1})$ have a projection larger than $M = M(S)$, a constant depending only on the topology of $S$.

        \item The transition from line (3) to (4) follows from the distance formula for the Teichm\"uller metric \cite{rafi2007combinatorial} (see \cite[Theorem 7.14]{durham2016augmented} or \cite[Proposition A.1]{dowdall2014statistical} for the version used here).

        \item Finally, the penultimate inequality follows from item (1) of Lemma \ref{lem:spacing}, where we replace $t_k$ with $t_{k-1}$ in line (6) depending (respectively) on whether $t \geq t_k$ or $t<t_k$.
    \end{itemize}

    In particular, $\gamma$ as $\log$-bounded projections, as required.
\end{proof}

\bibliography{refs}
\bibliographystyle{alpha}

\end{document}